\documentclass[12pt]{article}
\usepackage{bbm}
\usepackage{graphicx}
\graphicspath{ {./images/} }
\usepackage{wrapfig}
\usepackage[mathscr]{euscript}
\usepackage{subcaption}
\usepackage{hyperref}
\usepackage{epstopdf}
\usepackage{enumerate}
\usepackage{amsmath,amsfonts,amssymb,amsthm,epsfig,epstopdf,titling,url,array}
\usepackage{color}
\usepackage{framed}
\usepackage[utf8]{inputenc}
\usepackage[english]{babel}
\usepackage{xparse}
\usepackage{authblk}
\usepackage[margin=2.4cm]{geometry}
\usepackage{authblk}
\newtheorem{lem}{Lemma}[section]
\newtheorem{prop}{Proposition}[section]

\newtheorem{rem}{Remark}[section]

\newtheorem{maintheorem}{Theorem}

\date{}
\begin{document}
\title{Chebyshev's method applied to polynomials with two distinct roots}
\author[1]{Tarakanta Nayak 
    \footnote{ tnayak@iitbbs.ac.in}}
\author[1]{Pooja Phogat  \footnote{ poojaphogat174acad@gmail.com}}
\affil[1]{Department of Mathematics, 
		Indian Institute of Technology Bhubaneswar, India}
	\date{}
\maketitle
\begin{abstract}
The Julia set of the Chebyshev's method applied to polynomials with exactly two distinct roots is shown to be connected, and  its Fatou set is proved to be the union of attracting basins corresponding to the two roots. Further, if the two roots have the same multiplicity then the common boundary of the two immediate basins is proved to be a connected subset of the Julia set. 
\end{abstract}

\section{Introduction}
Finding the roots of a polynomial is a classical as well as an extensively studied problem.  
A \textit{root-finding method} is a function that assigns to each polynomial $p: \widehat{\mathbb{C}} \to \widehat{\mathbb{C}}$, a rational map $F_p$ such that every root $z_0$ of $p$ is an attracting fixed point of $F_p$, i.e., $F_p (z_0)=z_0$ and $|(F_p)'(z_0)|<1$. Here $\mathbb{C} \cup \{\infty\} $ is denoted by $ \widehat{\mathbb{C}}$. The forward orbit $\{F_p ^n (z): n \geq 1\}$ of a point $z$ \textit{sufficiently close} to a root of $p$ converges to that root of $p$.  
However, other points can behave very differently.  For some root-finding methods, there may exist polynomials and initial points whose forward orbits fail to converge or converge to a point that is \textit{not} a root of the polynomial. Although such situations are undesirable from the  root-finding perspective, these give rise to interesting dynamical phenomena.
This article is concerned with the dynamics of the Chebyshev's method applied to polynomials with two distinct roots.

\par By dynamics of an analytic function, we mean its Fatou and the Julia set.  
Given a rational function $R$ of degree at least two, the \textit{Fatou set} $\mathcal{F}(R)$ is the set of all points in $\widehat{\mathbb{C}}$ for which there exists a neighbourhood where the family of iterates $\{R^n\}_{n \geq 0}$ is equicontinuous. 
Its complement in $\widehat{\mathbb{C}}$ is the \textit{Julia set} $\mathcal{J}(R)$.  
By definition, $\mathcal{F}(R)$ is   open  while $\mathcal{J}(R)$, being its complement, is closed. A maximally connected component of $\mathcal{F}(R)$  is called a \textit{Fatou component}. Each Fatou component is either periodic or  some iterated image of it is periodic. A periodic Fatou component is an attracting domain, a parabolic domain, a Siegel disk or a Herman ring (see~\cite{Beardon_book} for further details).  
\par  For an attracting fixed point $z_0$, the set of all points  converging to $z_0$ under iteration of $R$ is called the \textit{basin of attraction} of $z_0$, i.e., the set
$\{z \in \widehat{\mathbb{C}} : \lim_{n \to \infty} R^n(z)=z_0 \}.
$
The connected component of the basin of attraction containing $z_0$ is the \textit{immediate basin} or \textit{attracting domain} and is denoted by $\mathcal{A}_{z_0}$.  
 \par
For a non-constant polynomial $p$, the \emph{Chebyshev's method} is defined by the rational map
$$
C_p(z)=z-\left(1+\frac{1}{2} L_p(z)\right) \frac{p(z)}{p^{\prime}(z)},
~\mbox{where}~
L_p(z)=\frac{p(z) p^{\prime \prime}(z)}{\left(p^{\prime}(z)\right)^2}
\text{.}$$
If $p$ is a linear polynomial, then $C_p$ is constant. The Chebyshev's method differs from the widely studied Newton method. The two significant differences are that Chebyshev's method exhibits \emph{third-order convergence} at each simple root of $p$, meaning its local degree at such roots is at least three, and that it can have a fixed point different from the roots of the polynomial. 
The dynamics of $C_p$ has been studied in \cite{Nayak-Pal2022}, where it is shown that if a cubic polynomial $p$ is unicritical (i.e., it has only one critical point) or non-generic (i.e., it has at least one multiple root), then the Julia set of $C_p $ is connected.  A \textit{critical point} of a rational map is a point where the local degree of the map is at least two. The derivative of the map vanishes at a finite critical point. A critical point $c^*$ of $p$ is said to be \textit{special} if $p(c^*)\neq 0$ but $p''(c^*)=0$. Suppose $p$ has $N$ distinct roots and $s$ distinct special critical points. Then the degree of $C_p$,  $\deg(C_p)=3N+s-B-2$, where $B$ denotes the sum of the multiplicities of all special critical points of $p$ (see Theorem~1.1, \cite{Nayak-Pal2022}). From this degree formula, it follows that if $p$ has at least three distinct roots then $\deg(C_p)\geq 6$, with equality only when $p$ is a cubic unicritical polynomial. Interestingly, there is no polynomial $p$ for which the degree of $C_p$ is five. The lowest possible  degree of $C_p$ is four, and it occurs precisely when $p$ has exactly two distinct roots. This motivates the study of the Chebyshev's method applied to polynomials with two distinct roots.
\par It has been proved by Nayak and Pal that if a polynomial is centered (i.e., with zero as the second leading coefficient) and has exactly two distinct roots with equal multiplicity, then the Julia set of $C_p$ is connected and $C_p$ is convergent (Theorem A,~\cite{Sym_dyn}), i.e., the Fatou set is the union of attracting basins corresponding to the roots of $p$. A key ingredient in their proof was that the critical points of $C_p$ (up to the Scaling property) are symmetric not only with  respect to the real line but also with respect to the origin. However, this is no longer true if the polynomial is not centered or the restriction on the multiplicity of roots is removed. In this article, we complete the dynamical study of the Chebyshev's method applied to all polynomials that are not necessarily centered  and having exactly two roots without any restriction on their multiplicities.  

\begin{maintheorem}\label{connected-convergent}
	If $p$ is a polynomial with exactly  two distinct roots,  then  the Julia set $\mathcal{J}(C_p)$   is connected and  $C_p$ is convergent, i.e., the Fatou set $\mathcal{F}(C_p)$ is the union of the basins of attractions corresponding to the roots of $p$.
	\end{maintheorem}
	
The Fatou and the Julia set of $C_p$ are illustrated in Figure~\ref{Julia_Fatou2}. A numerical study of several root-finding methods such as Newton, Chebyshev, Halley and Schröder (all for multiple roots),  applied to polynomials with two distinct roots is done in~\cite{CVV2024}.
	\begin{figure}[h!]
	\begin{subfigure}{.5\textwidth}
		\centering
		\includegraphics[width=0.975\linewidth]{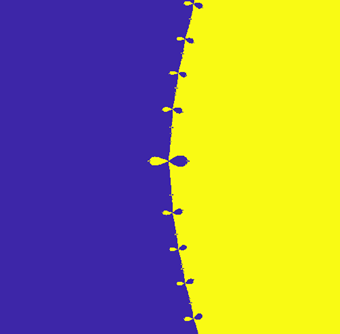}
		\caption{$k=6,m=4$}
	\end{subfigure}
	\begin{subfigure}{.5\textwidth}
		\centering
		\includegraphics[width=0.975\linewidth]{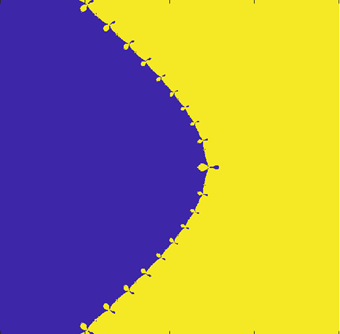}
		\caption{$k=3,m=10$}
	\end{subfigure}
	\caption{Fatou and Julia sets of $C_p$ where $p(z)=z^k (z-1)^m$. The blue and yellow regions constitute the basins of the $0$ and $1$ respectively. The largest regions (blue and yellow) are the immediate basins.}
	\label{Julia_Fatou2}
\end{figure}
 Note that for Newton's method, $\infty$ is the only extraneous fixed point, i.e., a fixed point  that is \textit{not} a root  of the polynomial, and it is repelling (see Section~\ref{prelim} for definition).  This fact along with Theorem~1 of \cite{Shishikura2009} leads to the conclusion that the Julia set of Newton's method applied to every polynomial is connected. The corresponding result is not yet established for the Chebyshev's method, though all known examples of Julia sets are connected. For   a polynomial $p$ with exactly two distinct roots, the degree of Newton's method $N_p$ is two. In this case, there are exactly two critical points of $N_p$, and they correspond to the basins of roots of $p$. The Julia set $\mathcal{J}(N_p)$ is a Jordan curve (in fact a straight line), and the Fatou set $\mathcal{F}(N_p)$ consists precisely of the two completely invariant components, namely the immediate basins of $N_p$ corresponding to the roots of $p$. This is also recently shown to be true for Halley's method  whenever the multiplicity of the two roots is the same (Theorem A, ~\cite{halley-liu-etal-2025}). However, this is not true for the Chebyshev's method. The number of Fatou components of $C_p$ can be infinite (for example, see Theorem A(1)~\cite{Sym_dyn}).

\par 
Connectedness of the Julia set implies that each Fatou component of $C_p$ is simply connected. In particular, the immediate basins corresponding to the roots of $p$ are simply connected. Moreover, for each immediate basin $\mathcal{A} $ corresponding to a root of $p$, the restriction  $C_p : \mathcal{A}  \to  \mathcal{A} $ is a proper map of degree three. This follows from the Riemann-Hurwitz formula (see Theorem 5.4.1, \cite{Beardon_book}). Since the degree of $C_p$ is four (see Equation~\ref{eq:Cp}), the immediate basin $\mathcal{A}$ is not completely invariant. Consequently, there are infinitely many Fatou components in addition to the two immediate basins.
 In particular, the Julia set of $C_p$ is not the common boundary of these two immediate basins. We investigate this common boundary in the case where the two roots of $p$ have equal multiplicity and establish the following.  

\begin{maintheorem}
	 	If $p$ is a polynomial with exactly two distinct roots  and the  multiplicities of these two roots are equal, then  the common boundary of the two immediate basins corresponding to the roots of $p$ is a connected subset of the Julia set of $C_p$.
 \label{equal-multiplicity}
\end{maintheorem}
\begin{figure}[h!]
	\begin{subfigure}{.5\textwidth}
		\centering
		\includegraphics[width=1.1 \linewidth]{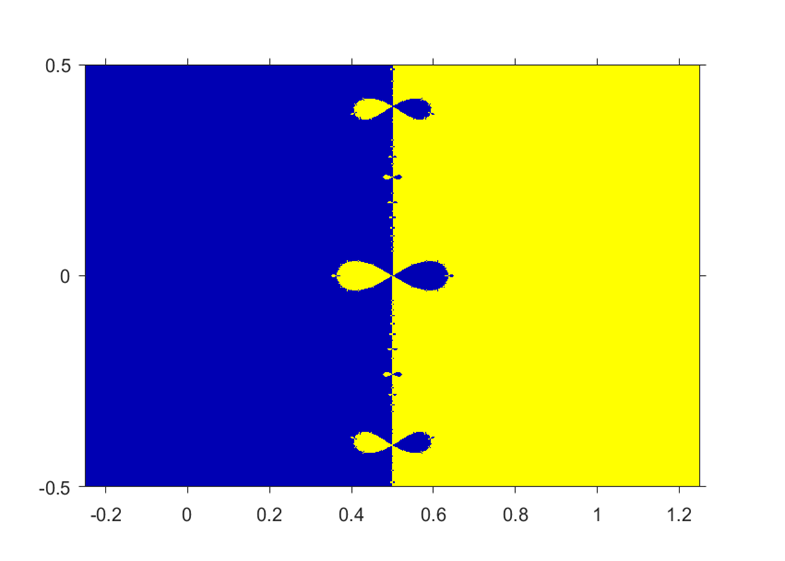}
		\caption{$m=2$}
	\end{subfigure}
	\begin{subfigure}{.5\textwidth}
		\centering
		\includegraphics[width=1.1 \linewidth]{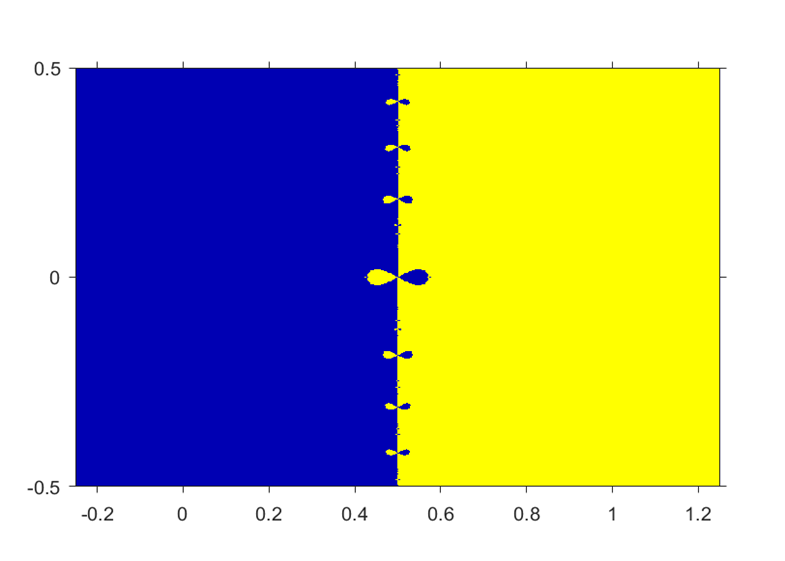}
		\caption{$m=7$}
	\end{subfigure}
	\vspace{0.5cm}
	\begin{subfigure}{.5\textwidth}
		\centering
		\includegraphics[width=1.1 \linewidth]{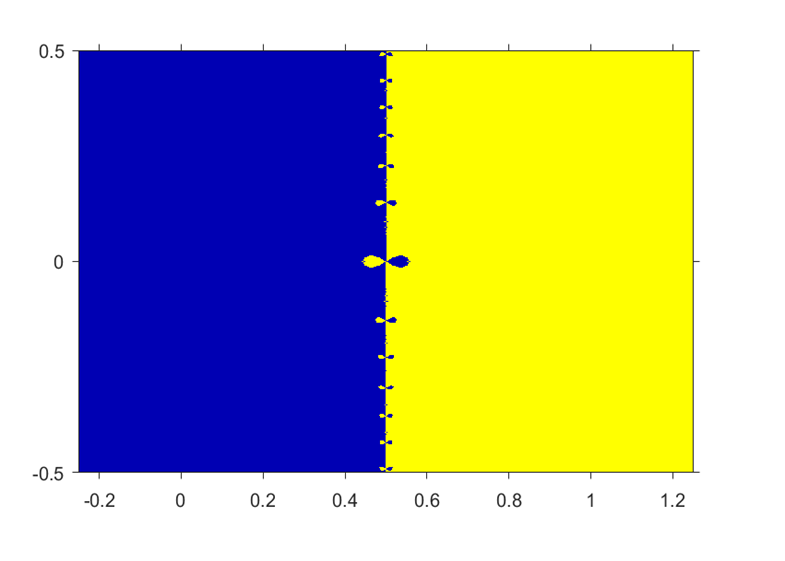}
		\caption{$m=12$}
	\end{subfigure}
	\begin{subfigure}{.5\textwidth}
		\centering
		\includegraphics[width=1.1 \linewidth]{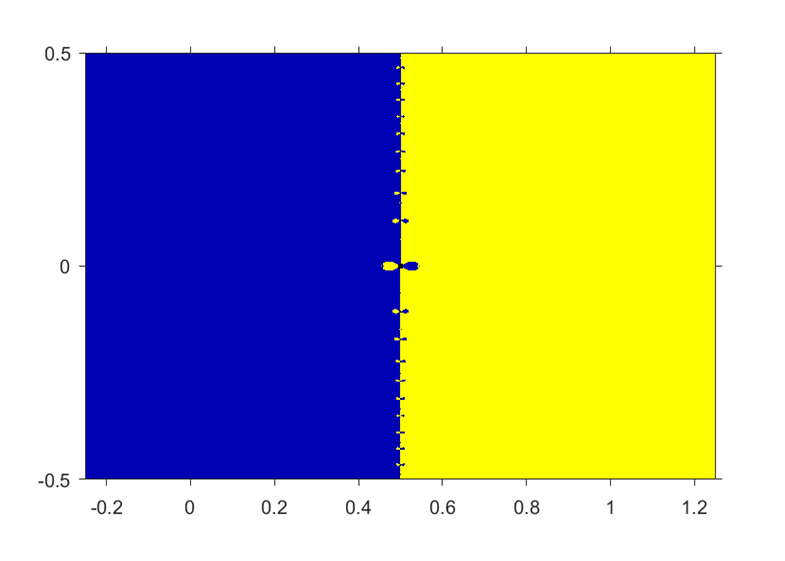}
		\caption{$m=20$}
	\end{subfigure}
	\caption{Fatou and Julia sets of $C_p$ for $p(z)=(z^2-z)^m$.}
	\label{k=m-dynamics}
\end{figure} 
 For the polynomial $p(z)=(z^2 -z)^m$, the immediate basins of $0$ and $1$ are represented by the largest  blue and yellow regions, respectively (see Figure~\ref{k=m-dynamics}). The common boundary of these regions in each case, is the vertical line passing through the pole $0.5$, which is clearly connected.
\par

This article is organised as follows. Section~\ref{prelim} contains all the preliminary results required later. The proofs of the results are provided in Section~\ref{proofs}.

\section{Preliminary results}
\label{prelim}
Two rational maps $R$ and $S$ are said to be \emph{conformally conjugate} (or \emph{conjugate}) if there exists a Möbius transformation $\phi$ such that
 $ S = \phi \circ R \circ \phi^{-1}, $ 
where $\circ$ denotes function composition. Since
 $ S^n = \phi \circ R^n \circ \phi^{-1} \quad \text{for all } n,$ 
the iterative behaviours of $R$ and $S$ are essentially the same. More precisely, we have the following result.

\begin{lem}[Theorem 3.1.4, \cite{Beardon_book}]
	If $S$ and $R$ are rational maps with $S = \phi \circ R \circ \phi^{-1}$ for some Möbius map $\phi$, then
$\mathcal{J}(S) = \phi(\mathcal{J}(R)).$
\label{conjugacy}
\end{lem}
The Chebyshev's method applied to different polynomials can give rise to the same rational map up to conjugacy. This is made precise by the \emph{Scaling property}.

\begin{lem}[Scaling property, Theorem 2.2, \cite{Nayak-Pal2022}]\label{scaling}
	Let $p$ be a polynomial of degree at least two. If $T(z) = \alpha z + \beta$ with $\alpha, \beta \in \mathbb{C}, \alpha \neq 0$, and $g = \lambda p \circ T$ for $\lambda \neq 0$, then 
$T \circ C_g \circ T^{-1} = C_p.
$
\end{lem}
Taking $T(z)=(b-a)z+a$ where $a,b$ are the only roots of $p$, and $\lambda$ as the reciprocal of the leading coefficient of $p \circ T$, a useful consequence of the Scaling property  follows. 
\begin{prop}If $p$ is a polynomial with exactly two distinct roots then $C_p$ is conformally conjugate to $C_q$ where $q(z)=z^k (z-1)^m$ for some natural numbers $k,m$.
\label{scaling-tworoots}	\end{prop}

%
%
Let $p$ be a polynomial with exactly two distinct roots with multilplicities $k$ and $m$.
It can be observed from Lemma~\ref{conjugacy} and Proposition~\ref{scaling-tworoots} that the Julia set of $C_p$ is connected and the Fatou set of $C_p$ is the union of the two attracting basins corresponding to the roots of $p$ if and only if it is so for $C_q$ where $q(z)=z^k (z-1)^m$. For proving Theorems~\ref{connected-convergent} and ~\ref{equal-multiplicity}, we assume  without loss of generality, that 
\begin{equation}
	p(z) =z^k(z-1)^m ~\mbox{for}~k,m \geq 1.
\label{generalform-polynomial}	
\end{equation}  Then  $p'(z)=z^{k-1}(z-1)^{m-1}((k+m)z-k), p''(z)  =z^{k-2}(z-1)^{m-2}f(z)~\mbox{and}~ p'''(z) =z^{k-3}(z-1)^{m-3}g(z)$, where  $f(z)=(k+m)(k+m-1)z^2-2k(k+m-1)z+k(k-1) $ and $g(z)=(k+m)(k+m-1)(k+m-2)z^3-3k(k+m-1)(k+m-2)z^2+3k(k-1)(k+m-2)z-k(k-1)(k-2)$. 
Note that  these expressions remain valid even for $k,m \leq 2$. Consequently, 
$	L_p(z)   =\frac{f(z)}{\left(kz+m z-k\right)^2},$
and \begin{equation}
	C_p(z)   = z-z(z-1)\left(\frac{(k+m)\left(3 k+3m -1\right) z^2
		-2k\left(3 k+3m -1\right) z+3k^2-k}
	{2\left(kz+mz-k\right)^3}\right)  \label{eq:Cp}.
\end{equation}
We can also write 
\begin{equation}
C_p (z) =\frac{A_0 z^4+A_1 z^3+A_2 z^2+A_3 z}
{2\left((k+m)z-k\right)^3}, \label{eq:Cp2} \end{equation}
where  $A_0=(k+m)(k+m-1)(2(k+m)-1)$, $A_1=(3-6k)(k+m)^2+(6k-1)(k+m)-2k $, $A_2= 3k(k-1)(2(k+m)-1)$, and $A_3= -k(k-1)(2k-1)$. An immediate consequence of Equation~(\ref{eq:Cp}) is the following.
\begin{lem} 
  If $c$ is a non-real critical point of $C_p$, then $\overline{c}$ is also a critical point of $C_p$. Moreover, every Fatou component of $C_p$ containing a real number is symmetric with respect to the real line.
 \label{symmetry}
\end{lem}
\begin{proof}
	For every $n \geq 1$ and $z\in \mathbb{C}$,  $C_p^n(\bar{z})=\overline{C_p^n(z)}$. Taking $n=1$ and differentiating it, yields $\overline{C_p '(\overline{z})}=C_p '(z)$, and hence $C_p '(c)=0$ if and only if  $C_p '(\overline{c})=0$. It also follows that the Fatou set is symmetric with respect to the real line. In particular, every Fatou component containing a real number is symmetric with respect to the real line.
\end{proof}

For stating some useful information about the fixed points of the Chebyshev's method, we recall relevant definitions.

A point $z_0 \in \widehat{\mathbb{C}}$ is a \textit{fixed point} of a rational map $R$ if $R(z_0)=z_0$.  
The \textit{multiplier} $\lambda_{z_0}$ of a fixed point $z_0$ is defined as $R'(z_0)$ if $z_0$ is finite and as $S'(0)$ if $z_0=\infty$, where $S(z)=\tfrac{1}{R(1/z)}$. A fixed point $z_0$ is called attracting, neutral or repelling if $|\lambda_{z_0}|<1, =1$ or $>1$, respectively.
An attracting fixed point is called \textit{super-attracting}  if its multiplier is $0$. 
	\begin{lem}[Proposition 2.3, \cite{Nayak-Pal2022}]
		If $a$ is a root of a polynomial $p$ with multiplicity $k$, then it is a fixed point of $C_p$ with multiplier $\frac{(k-1)(2k-1)}{2k^2}$. Moreover, the point at $\infty$ is also a fixed point of $C_p$ and its multiplier is $\frac{2d^2}{2d^2 -3d+1}$, where $d$ is the degree of $p$.
		\label{multiplier-fixedpoints}
	\end{lem}
We need the following two well-known results in complex dynamics.  

\begin{lem} (\cite{Beardon_book}) \label{cpoint}
  Let $U$ be a periodic Fatou component of a rational map $R$.  
 If $U$ is an immediate attracting or parabolic basin, then $U$ contains at least one critical point of $R$. If $U$ is a Siegel disk or a Herman ring, then the boundary of $U$ is contained in the closure of the post-critical set   $\{R^n (c): R'(c) =0 ~\mbox{and}~ n \geq 0\}$.
\end{lem}

\begin{lem}(Lemma 4.3, \cite{Nayak-Pal2022}) \label{Jconnected}
	Let $R$ be a rational map for which $\infty$ is a repelling fixed point. If $\mathcal{A}$ is an unbounded invariant immediate basin of attraction, then its boundary contains at least one pole of $R$. Furthermore, if all the poles of
	$R$  lie on the boundary of $\mathcal{A}$ and $\mathcal{A}$ is simply connected, then the Julia set of $R$ is connected.
\end{lem}
The behaviour of analytic branches of the inverse of the iterates of a rational function on a domain intersecting its Julia set is going to be useful for this article. We put Theorem~9.2.1  and Lemma~9.2.2 of ~\cite{Beardon_book} together to make this precise.
\begin{lem}
For a rational function $R$, let $f_k$ be a single-valued  analytic branch of $R^{-k}$ defined in a domain $D$. If $D$ intersects the Julia set of $R$, then the family $\{f_k\}_{k>0}$ is normal and each of its uniform limits is constant.	
	\label{backward-contraction}
\end{lem}

\section{Proofs of results}
\label{proofs}
In view of the discussion preceding Equation~(\ref{generalform-polynomial}), it is enough to prove Theorems ~\ref{connected-convergent} and ~\ref{equal-multiplicity} for $p(z)=z^k (z-1)^m$ for $k,m \geq 1$. Recall that 
$$C_p(z)   = z-z(z-1)\left(\frac{(k+m)\left(3 k+3m -1\right) z^2
	-2k\left(3 k+3m -1\right) z+3k^2-k}
{2\left(kz+mz-k\right)^3}\right).$$
The next four lemmas discuss the critical points, extraneous fixed points and the behaviour of $C_p$ on the real line.

\begin{lem}
There is no critical point of $C_p$ in $(-\infty, 0) \cup (1, \infty)$. Furthermore, $C_p '(x)>0$ for all $x \in (-\infty, 0) \cup (1, \infty)$. 
\label{criticalpointsC_p}
\end{lem}
\begin{proof}
Note that $$\begin{aligned}
	C_p^{\prime}(z) &=\frac{1}{2}\left(\frac{p(z)p''(z)}{\left(p'(z)\right)^2}\right)^2\left(3- \frac{p'(z)p'''(z)}{\left(p''(z)\right)^2}\right) =\left(p(z)\right)^2\left(\frac{\left(3\left(p''(z)\right)^2-p'(z)p'''(z)\right)}{2(p'(z))^4}\right)\\
	&=z^{2k}(z-1)^{2m}\left(\frac{3z^{2k-4}(z-1)^{2m-4}\left(f(z)\right)^2-z^{2k-4}(z-1)^{2m-4}\left((k+m)z-k \right)g(z)}{2(p'(z))^4}\right)\\
	&=\frac{3\left(f(z)\right)^2-\left(kz+mz-k \right)g(z)}{2\left(kz+mz-k \right)^4},
\end{aligned}
$$ where the polynomials $f,g $
 are as defined just after Equation~(\ref{generalform-polynomial}).
In other words,  $$C_p'(z)=\frac{F(z)}{2\left(kz+mz-k \right)^4},$$  where  
$F(z)=(k+m)^2 (k+m-1)(2k+2m-1)z^4 +4k(k+m)(k+m-1)(-2 k-2m+1)z^3+ 
3k(k(k+m-1)(3 k+3m -2)+(k-1)(k+m)^2)z^2 -2k(k-1)((4k+1)(k+m)-3k)z+k^2(k-1)(2k-1)$.
The critical points of  $C_p$ are $\frac{k}{k+m}$ (counted twice) and the four roots of the quartic polynomial $F$.

Note that the coefficients of the odd powers of $z$ in $F(z)$ are always negative, whereas those of the even powers are positive. Therefore, $F(z)$ is   positive for all $z \in (-\infty,0)$. Hence $C_p$ has no critical point in $(-\infty,0)$.
\par  In order to examine the behaviour of $F$ on  $(0,\infty)$, we consider
$F(x+1)=(k+m)^2(k+m-1)(2(k+m)-1)x^4+4m(k+m)(k+m-1)(2(k+m)-1)x^3+3m(k^2(4m-1)+km(8m-7)+2m(2m^2-3m+1))x^2+2m(m-1)(k(4m+1)+2m(2m-1))x+m^2(2m^2-3m+1)$.
For $k,m \geq 1$, each coefficient  of $F(x+1)$ is either zero or  positive. Moreover, the coefficient of $x^2$ is strictly greater than $1$. This gives that $F(x+1)>1$ for every $x>0$. Since every real number in $(1,\infty)$ can be written as $1+x$, for some $x>0$, it follows that $F((1,\infty))\subset (1,\infty)$. In other words,  $F$ has no zero in $(1,\infty)$, i.e., $C_p$ has no critical point in $(1,\infty)$.
\par It also follows from the two previous paragraphs that $C_p '(x)>0$ for all $x \in (-\infty, 0) \cup (1,  \infty)$.
\end{proof}
It follows from Equation $(\ref*{eq:Cp})$ that the extraneous fixed points of $C_p$ are precisely the roots of  $E(z)=(k+m)\left(3k+3m-1\right) z^2-2k\left(3k+3m-1\right) z+3k^2-k$. 
These are given explicitly by
$\frac{2k\left(3k+3m-1\right)\pm \sqrt{(2k\left(3k+3m-1\right))^2-4(k+m)\left(3k+3m-1\right)(3k^2-k)}}{2(k+m)\left(3k+3m-1\right))}$, which upon simplification is nothing but,
$$
 \frac{k\left(3k+3m-1\right)\pm \sqrt{km\left(3k+3m-1\right)}}{(k+m)\left(3k+3m-1\right))}.
 $$
 
Since $km\left(3k+3m-1\right) > 0$ for all $k,m \geq 1$, both of these extraneous fixed points are real and distinct. Denoting the smaller and the larger extraneous fixed points by $e_1$ and $e_2$  respectively, we have\\
$$
\begin{aligned}
	e_1&=\frac{k\left(3k+3m-1\right)- \sqrt{km\left(3k+3m-1\right)}}{(k+m)\left(3k+3m-1\right))}=\frac{k}{k+m}-\frac{\sqrt{km\left(3k+3m-1\right)}}{(k+m)\left(3k+3m-1\right)}~\mbox{and}\\
	e_2&=\frac{k\left(3k+3m-1\right)+ \sqrt{km\left(3k+3m-1\right)}}{(k+m)\left(3k+3m-1\right))}=\frac{k}{k+m}+\frac{\sqrt{km\left(3k+3m-1\right)}}{(k+m)\left(3k+3m-1\right)}.
\end{aligned}
$$
Since $\frac{k}{k+m}$ is the pole of $C_p$ and $\frac{\sqrt{km\left(3k+3m-1\right)}}{(k+m)\left(3k+3m-1\right)}>0$ for all $k,m \geq 1$, the points $e_1$ and $e_2$ lie on opposite sides of the pole and are equidistant from it. More information on extraneous fixed points is provided below.
 
\begin{lem}
	Both extraneous fixed points of $C_p$ lie in the interval $(0,1)$,  and their multipliers are greater than $1$.  In particular, both are repelling.
	\label{extraneous}
\end{lem}
\begin{proof}
		Since $e_1<e_2$, it suffices to show that $e_1>0$ and $e_2<1$ for all $k,m \geq 1$.\\
	First, $e_1>0$ is equivalent to  $\frac{k}{k+m}-\frac{\sqrt{km\left(3k+3m)-1\right)}}{(k+m)\left(3k+3m-1\right)}>0,$ which implies
	$\frac{k}{k+m}>\frac{\sqrt{km\left(3k+3m-1\right)}}{(k+m)\left(3k+3m-1\right)}.$ Squaring both sides gives
	$\left(\frac{k}{k+m}\right)^2>\frac{km\left(3k+3m-1\right)}{(k+m)^2\left(3k+3m-1\right)^2},$ which simplifies to
	$k>\frac{m}{\left(3k+3m-1\right)}.$ Equivalently,
	$(3k-1)(k+m)>0 $ and this inequality holds for all $k,m \geq 1$. Thus $e_1>0$.\\
	Similarly,
	$e_2<1$ is equivalent to $\frac{k}{k+m}+\frac{\sqrt{km\left(3k+3m-1\right)}}{(k+m)\left(3k+3m-1\right)}<1,$ which implies
	$\frac{\sqrt{km\left(3k+3m-1\right)}}{(k+m)\left(3k+3m-1\right)}<\frac{m}{k+m}.$ Squaring both sides gives
	$\frac{km\left(3k+3m-1\right)}{(k+m)^2\left(3k+3m-1\right)^2}<\left(\frac{m}{k+m}\right)^2,$ or equivalently
	$\frac{k}{\left(3k+3m-1\right)}<m.$ This can be written as
	$(3m-1)(k+m)>0,$  which is actually
	true for all $k,m \geq 1.$ Therefore $e_2<1$.
	
	\par 
	Consider $g(z)=(z-1)^k(z+1)^m$. Let $T(z)=\frac{-z+1}{2}$ so that $T^{-1}(z)=-2z+1$. Then $g(z)=\lambda p(T(z))$, where $\lambda=(-2)^{k+m}$. By Lemma \ref{scaling}, $T \circ C_g \circ T^{-1}=C_p$. It is proved in  Lemma 3.1, \cite{Sym_dyn} that the multipliers of the two extraneous fixed points of $C_g$ are
	
	$1+\frac{(3k+3m-1)^2}{km(k+m)^2} \left[km (\frac{3k+3m-2}{3k+3m-1})\pm(k-m)\sqrt{\frac{km}{3k+3m-1}} \right]$. It can be seen that each of these multplier is greater than one.
	\par  The extraneous fixed points of $C_g$ are mapped to those of $C_p$ by the conjugating map $T$. 
The multiplier of a fixed point remains unchanged under conformal conjugacy (see page 36 of \cite{Beardon_book}). This  completes the proof.
\end{proof}
\begin{rem}
It follows that $0<e_1<\xi<e_2<1$, where $\xi=\frac{k}{k+m}$ is the pole of $C_p$.
\label{extraneous-rem}
\end{rem}
 
It is immediate from Equation (\ref*{eq:Cp}) that
\begin{align}
	C_p(x)-x &=-x(x-1)\left(\frac{(k+m)\left(3k+3m-1\right) x^2-2k\left(3k+3m-1\right)x+3k^2-k}{2\left((k+m)x-k\right)^3}\right) \nonumber \\
	&=\frac{-x(x-1)E(x)}{2\left((k+m)x-k\right)^3}, ~\mbox{where}~  E(x)= (k+m)\left(3k+3m-1\right)(x-e_1)(x-e_2).
  \label{eq:CpR}
\end{align}
 Here $e_1$ and $ e_2 $ are extraneous fixed points of $C_p$. We list some useful information on $C_p$ that follow from Equation~(\ref{eq:CpR}) and the facts that $E(x)>0$ for $x \in (-\infty,0)\cup (1,\infty)$ and, $2\left((k+m)x-k\right)^3$ is negative and positive for $x<0$ and $x>1$, respectively (see Remark~\ref{extraneous-rem}).
 
 \begin{lem} Let $ e_1$ and $ e_2$ be the extraneous fixed points and $\xi$ be the pole of $C_p$. Then the following holds.
 	\begin{enumerate}
 		\item For $x \in (-\infty,0)$ and $x \in (1,\infty)$, we have $ C_p(x)>x$ and $C_p (x)<x$, respectively.
 		\item For $x \in (0, e_1)$ and $x \in (e_1, \xi)$, we have $ C_p(x)<x$ and $C_p (x)>x$, respectively.
 		\item $\lim_{x \to {\xi}^-}C_p(x)=\infty  ~~\text{and}~~\lim_{x \to {\xi}^+}C_p(x)=-\infty.$
 	\end{enumerate}
 	\label{C_p-real}
 \end{lem}   
 We also need a lemma describing the mapping behaviour of $C_p$ on the real line.
 \begin{lem}  For $x<0$ and $x>1$, we have $ C_p (x)<0$ and $C_p (x)>1$, respectively.
 \end{lem}
 \begin{proof}
 	 From Equation (\ref{eq:Cp2}), note that the numerator of $C_p(z)$ has positive coefficients of all the even powers and negative coefficients of all the odd powers  of $z$. This ensures that for $x \in (-\infty,0)$, the numerator is positive while the denominator is negative. Therefore, $C_p(x)<0$ for each $x < 0$. Moreover, $C_p(1)=1,C_p(\infty)=\infty$, and $C_p$ has no critical point in $(1,\infty)$ by Lemma~\ref{criticalpointsC_p}. This implies that $C_p$ is strictly increasing in $(1, \infty)$. In other words, $C_p (x)>1$ for each $x>1$ (see Figure \ref{Graph} and Figure~\ref{graph-phi} (b)).
 \end{proof}

\begin{figure}[h!]
	\begin{subfigure}{.5\textwidth}
		\centering
		\includegraphics[width=1\linewidth]{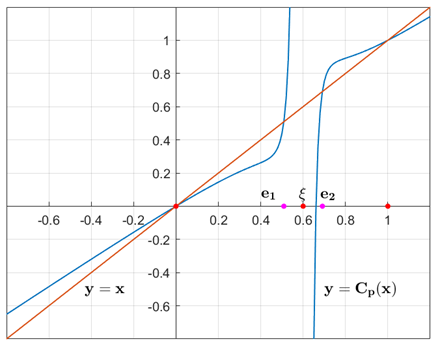}
		\caption{$k=6,m=4$}
	\end{subfigure}
	\begin{subfigure}{.5\textwidth}
		\centering
		\includegraphics[width=1\linewidth]{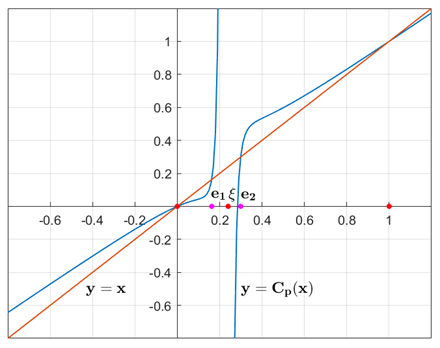}
		\caption{$k=3,m=10$}
	\end{subfigure}
	\caption{The graphs of $C_p: \mathbb{R} \to \mathbb{R}$ for $p(z)=z^k(z-1)^m.$}
	\label{Graph}
\end{figure}
\subsection{Proof of Theorem~\ref{connected-convergent}}
We require some additional information on the immediate basins for the proof of Theorem~\ref{connected-convergent}.
\begin{lem}
	The immediate basins $\mathcal{A}_0$ and $\mathcal{A}_1$  contain $(-\infty,e_1)$ and $(e_2, \infty)$, respectively. Furthermore, each such immediate basin contains  two critical points, counting multiplicity. 
	\label{criticalpoints-immediatebasins}
\end{lem}
\begin{proof}
It is already found that $C_p(x)>x$  (Lemma~\ref{C_p-real}(1)) and $C_p'(x)>0$  (Lemma~\ref{criticalpointsC_p}) for all $x< 0$. Since $C_p(0)=0$, the sequence $\{C_p^n(x)\}_{n>0}$ is strictly increasing and bounded above by $0.$ This implies that  $\lim_{n \to \infty}C_p^n(x)=0 $ for all $x<0$.
  Since $C_p(x)<x$ for all $x \in (0,e_1)$ (see Equation (\ref{eq:CpR})), there are two possibilities for the sequence $\{C_p ^n (x)\}_{n>0}$: either it converges to $0$ or there exits $n_0$ such that $C_p ^{n_0} (x) \leq 0$. In either case, $(0, e_1) \subset  \mathcal{A}_0$.  
Therefore,  $$(-\infty,e_1)\subset \mathcal{A}_0.$$ 
	
Similarly, we have $C_p(x)<x$ and $C_p'(x)>0$ for all $x>1$. This implies that for all $x$ in $(1,\infty)$, $\lim_{n \to \infty}C_p^n(x)=1.$
 Using a similar argument as before and the fact that $C_p (x)> x$ for all $x \in (e_2, 1)$, we have   $$(e_2,\infty)\subset \mathcal{A}_1.$$

\par
Let $\mathcal{A}$ be  the immediate basin of either $0$ or of $1$. Then it contains a critical point $c$ by Lemma~\ref{cpoint}. The proof will be complete by showing that   $\mathcal{A}$ contains at least two critical points counting multiplicity. This is because, there are only four critical points available for  the two immediate basins of $C_p$. 
\par If the  fixed point corresponding to $\mathcal{A}$ is super-attracting, then the super-attracting fixed point is itself a critical point with multiplicity at least two (see   Equation~(\ref{eq:Cp2})) and we are done.
\par  Assume that the fixed point corresponding to $\mathcal{A}$ is not super-attracting. If $c$ is non-real,  then   $\overline{c}$  is  also a critical point of $C_p$ and lies in $\mathcal{A}$, by Lemma~\ref{symmetry}. Again we are done. For real  $c$, we give a proof for   $\mathcal{A}=\mathcal{A}_0$. This proof works for  $\mathcal{A}=\mathcal{A}_1$.
\par Note that   $c \in (0,e_1)$ by Lemma~\ref{criticalpointsC_p} and the fact that $\mathcal{A}_0 \cap \mathbb{R}=(-\infty, e_1)$.  Clearly  $ C_p'(0)=\frac{(k-1)(2k-1)}{2 k^2} \in (0,1)$ and $ C_p'(e_1)>1$ (see Lemma~\ref{extraneous}) If $c$ is a simple root of $C_p'$, i.e., $C_p'(c)=0$ and $C_p''(c)\neq 0$ then $C_p''(c)>0$ or $<0$. For $C_p''(c)>0$,  the function $C_p'$ is strictly increasing in an interval containing $c$ and it follows from the Intermediate Value Theorem applied to $C_p' : [0,e_1] \to \mathbb{R}$ that  there  exists a  root of $C_p '$ in $(0,c)$. Similarly, it can be seen that $C_p'$ has a root in $ (c, e_1)$ whenever  $C_p ''(c) <0$. (refer to Figure~\ref{possiblegraphs}). If $c$ is a multiple root of $C_p '$ then its multiplicity must be at least two. Therefore, the immediate basin $\mathcal{A}$ contains at least two critical points whenever it contains a real critical point, i.e., $c$ is real.
 
%
%
\end{proof}
\begin{figure}[h!]
	\begin{subfigure}{.5\textwidth}
		\centering
		\includegraphics[width=1\linewidth]{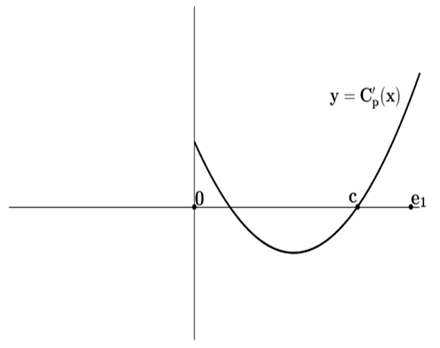}
		\caption{$C_p ''(c)>0$}
	\end{subfigure}
	\begin{subfigure}{.5\textwidth}
		\centering
		\includegraphics[width=1\linewidth]{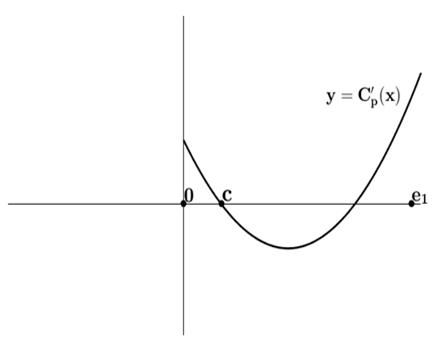}
		\caption{$C_p ''(c)<0$}
	\end{subfigure}
	\caption{The possible graphs of $C_p ': [0,e_1]  \to \mathbb{R} $.}
	\label{possiblegraphs}
\end{figure}
\begin{proof}[Proof of Theorem~\ref{connected-convergent}]
Both $\mathcal{A}_0$ and $\mathcal{A}_1$ are invariant immediate basins of attraction,  and $\infty$ is a repelling fixed point of $C_p$ (its multiplier is $\frac{2(k+m)^2}{2(k+m)^2 -3(k+m)+1}$ by Lemma~\ref{multiplier-fixedpoints}). It follows from Lemma \ref{Jconnected} that both $\mathcal{A}_0$ and $\mathcal{A}_1$ have at least one pole on their respective boundaries. Since $C_p$ has exactly one pole, namely $\frac{k}{k+m}$, this pole in  the intersection $ \partial \mathcal{A}_0 \cap \partial \mathcal{A}_1 $. 
	
	If $\mathcal{A}_0$ is not simply connected, then there exists a Jordan curve $\gamma$ in $\mathcal{A}_0$, surrounding  a bounded Julia component (intersecting $\partial \mathcal{A}_0$). Here, by a \textit{Julia component}, we mean a maximally connected subset of the Julia set. Since $\infty$ is in the Julia set of $C_p$ and the set $\{z: C_p ^n (z)=\infty~\mbox{for some}~n>0\}$ is dense in the Julia set, $C_p^l(\gamma)$ surrounds the pole $\frac{k}{k+m}$ for some $l$. This is however not possible, as  $C_p^l(\gamma) \subset \mathcal{A}_0$ (since $\mathcal{A}_0$ is invariant), $\frac{k}{k+m} \in \partial \mathcal{A}_0 \cap \partial \mathcal{A}_1$  and $\mathcal{A}_1$ is unbounded. Therefore, $\mathcal{A}_0$ must be simply connected. It follows from the second part of Lemma \ref{Jconnected} that the Julia set of $C_p$ is connected.
	\par 
	The pole $\frac{k}{k+m}$ is a critical point  of $C_p$ with multiplicity two and is mapped to $\infty$, which lies in the Julia set. There are four other critical points of $C_p$ and each of the immediate basins $\mathcal{A}_0$ and $\mathcal{A}_1$
 contains exactly two critical points by Lemma~\ref{criticalpoints-immediatebasins}. It now follows from Lemma~\ref{cpoint} that the Fatou set of $C_p$ is the union of these immediate basins and their iterated pre-images. In other words, $C_p$ is convergent. 
\end{proof}
The case $k=1$ is described in the following remark.
\begin{rem}
For $k=1$, $$F(z) =z^2 \left( m(m+1)^2(2m+1)z^2-4m(m+1)(2m+1)z+
9m^2+3m \right).$$ It does not have any  non-zero real root. In fact, $F(x)\geq 0$ for all $x \in \mathbb{R}$. This gives $C_p'(x)>0$   and $C_p$ is strictly increasing in $\mathbb{R} \setminus \{0, \frac{1}{m+1}\}$. The finite critical points of $C_p$ are $0,0,\frac{1}{m+1},\frac{1}{m+1},\frac{1}{m+1}\left(2\pm  i \sqrt{\frac{m-1}{2m+1}}\right)$. The critical point $0$ is itself fixed, the multiple pole $\frac{1}{m+1}$ is mapped to $\infty$, and $\infty$ is fixed.  The fixed points of $C_p$ are  precisely $0,1$ and the roots of the equation $(3m+2)(m+1)z^2-2(3m+2)z+2=0$, which are $\frac{1}{m+1}(1\pm \sqrt{\frac{m}{3m+2}})$ and are the extraneous fixed points of $C_p$. On the real line, graphs of $C_p(x)$ for $m=2,3$ are given in Figure~\ref{Graph-k=1}. The immediate basins $\mathcal{A}_0$ and $\mathcal{A}_1$ for $k=1$ are illustrated in blue and yellow, respectively in Figure~\ref{k=1-dynamics}.
\begin{figure}[h!]
\begin{subfigure}{.5\textwidth}
\centering
\includegraphics[width=1\linewidth]{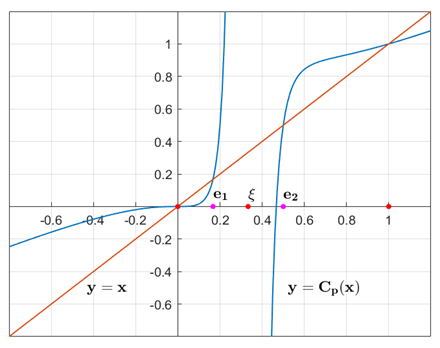}
\caption{$m=2$}
\end{subfigure}
\begin{subfigure}{.5\textwidth}
\centering
\includegraphics[width=1\linewidth]{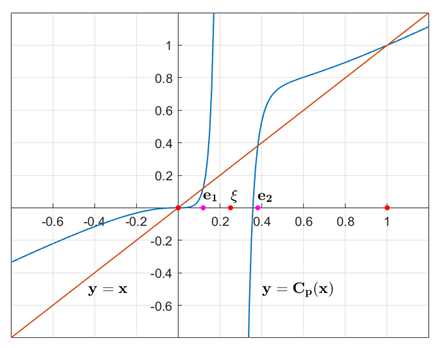}
\caption{$m=3$}
\end{subfigure}
\caption{The graphs of $C_p: \mathbb{R} \to \mathbb{R}$ for $p(z)=z(z-1)^m.$ }
\label{Graph-k=1}
\end{figure}
\begin{figure}[h!]
	\begin{subfigure}{.5\textwidth}
		\centering
		\includegraphics[width=0.975\linewidth]{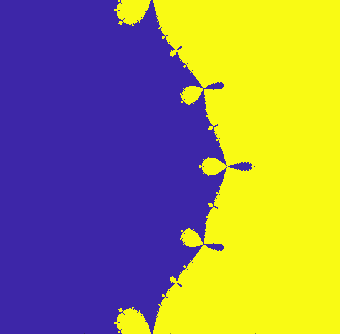}
		\caption{$m=2$}
	\end{subfigure}
	\begin{subfigure}{.5\textwidth}
		\centering
		\includegraphics[width=0.975\linewidth]{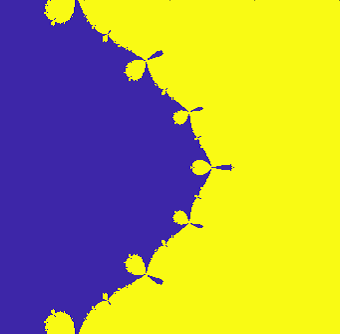}
		\caption{$m=3$}
	\end{subfigure}
	\caption{The Fatou and Julia sets of $C_p$ for $p(z)=z (z-1)^m$.}
	\label{k=1-dynamics}
\end{figure}
\end{rem}
\subsection{Proof of Theorem ~\ref{equal-multiplicity}}
Let $p_m (z)=z^m (z-1)^m$ and $C_m $ denote the Chebyshev's method applied to $p_{m}$.  
Let $L$ be the vertical line passing through $0.5$, the unique pole of $C_m$. Then,  \begin{equation}
	C_m (0.5 +iy)= 0.5 + i\frac{16 (2m -1)(4m -1)y^4 -24 m y^2 -1}{128 m^2 y^3}=0.5+i\phi_{m}(y), say. 
	\label{phi-expression}\end{equation} 

The behaviour of $C_m : L \to L$ can be analysed via  $\phi_m: \mathbb{R} \to \mathbb{R}$ which is defined in Equation~(\ref{phi-expression}). For brevity, let $C_m =C$ and $\phi_{m} =\phi$ as long as the proof of Theorem~\ref{equal-multiplicity} is concerned. We now list some useful properties of $\phi$.

	\begin{figure}[h!]
	\begin{subfigure}{.5\textwidth}
		\centering
		\includegraphics[width=1.0 \linewidth]{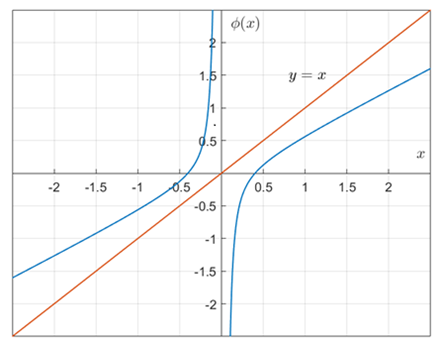}
		\caption{The graph of $\phi: \mathbb{R} \to \mathbb{R}$ for $m=2$.}
	\end{subfigure}
	\begin{subfigure}{.5\textwidth}
		\centering
		\includegraphics[width=1.0 \linewidth]{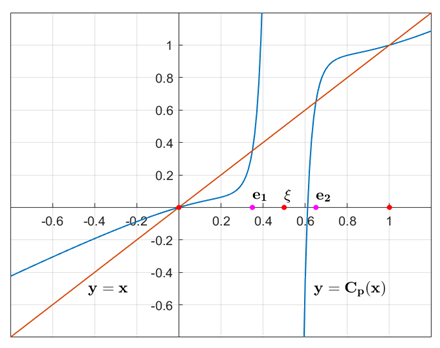}	
		\caption{ The graph of $C_{p_m}: \mathbb{R} \to \mathbb{R}$ for $m=2$.}	\end{subfigure}
	\caption{Graphs of $\phi$ and $C_{p_m}$.}
	\label{graph-phi}
\end{figure}
\begin{lem}
	Let  $m \geq 1$ and $\phi : \mathbb{R} \to \mathbb{R}$ be defined by $\phi (x)=\frac{16 (2m -1)(4m -1)x^4 -24 m x^2 -1}{128 m^2 x^3}$. Then the following statements are true. 
	\begin{enumerate}
		\item The map $\phi	$ is odd,  and $\phi(x)< x$ for all $x>0$.  In particular, it does not have any fixed point.  
		\item The map $\phi	$ is  strictly increasing in $\mathbb{R} \setminus \{0\}$, and   $\lim_{x \to +\infty} \phi(x) = +\infty$ and $\lim_{x \to 0^+} \phi(x) = -\infty$. In particular, $\phi$ has a unique positive zero, say $ \zeta$.
		
		\item For each $x>\zeta$, there is  $n_x \geq 1$ such that $\phi^{n_x}(x) \in (0, \zeta]$. Further, $n_x$ can be chosen as the smallest such number.  
		\item For every real number $x$,   $\phi^n (x) \in (0, \zeta)$ for infinitely many values of $n$ unless $\phi^n (x)= 0$ for some $n \geq 0$.
		
	\end{enumerate}
	\label{phi}
\end{lem} 
\begin{proof} Note that $\phi(x)=\frac{ (2m -1)(4m -1)}{8m^2}x-\frac{3}{16m }\frac{1}{x} -\frac{1}{128m^2 }\frac{1}{x^3}.$
\begin{enumerate}
\item 
		It is clear from the expression of $\phi$ that it is odd and $\phi(x)<x$ for all $x>0$. Since $\phi$ is odd, we also have $\phi(x) > x$ for all $x<0$. In particular, $\phi$ does not have any fixed point.
\item Since $\phi '(x)>0$ for all $x \neq 0$, $\phi$ is strictly increasing in $\mathbb{R} \setminus \{0\}$. It follows from the expression of $\phi$ that  $\lim_{x \to +\infty} \phi(x) = +\infty$ and $\lim_{x \to 0^+} \phi(x) = -\infty$. Clearly, $\phi$ has a unique positive zero, say $\zeta$.
\item Let $x> \zeta$. Then    $0<\phi (x)<x$. If
		$\phi (x)   \in (0,\zeta]$, then we are done. Otherwise,  $\zeta < \phi^2 (x)< \phi(x)$. The last inequality $\phi^2 (x)< \phi(x)$ results from the strict increasingness of $\phi$.  Repeating this argument for $\phi^2 (x)$, it is seen that either $\phi^2 (x) \in (0, \zeta]$ or  $\zeta < \phi^3 (x)< \phi^2(x) < \phi(x)$. This process can not be continued indefinitely because then $\{\phi^n (x)\}_{n>0}$ would become a strictly decreasing sequence that is bounded below by $\zeta$ making it convergent and the limit point has to be a fixed point of $\phi$ in $[\zeta, x)$.  However, this is not possible by (1) of this lemma. Therefore, for each $x>\zeta$,   there is  $n_x \geq 1$ such that $\phi^{n_x}(x) \in (0, \zeta]$.
		\item 
		Let there be a real $x$ such that $\phi^n (x) \neq 0$ for any $n \geq 0$. Then  $\phi^n (x) \neq \zeta~\mbox{or}~-\zeta$ for any $n \geq 0$. Denote $(0, \zeta)$ by $I_\zeta$.  First we show that  $\phi^n (x) \in I_\zeta$ for infinitely many values of $n$ for each $x \in I_\zeta$. The  proof will then be complete by showing that for each real $x \notin I_\zeta$, $\phi^m (x) \in I_\zeta$ for some $m$.

\par  If $x \in I_\zeta$ then $\phi(x)<0$ but  $\phi(x) \neq -\zeta$. For $\phi(x) < -\zeta$, there is $n_{\phi(x)}$ such that $\phi^{n_{\phi(x)}} (\phi(x)) \in (-\zeta,0)$. This follows from (3) of this lemma and the fact that $\phi$ is an odd function.   If $\phi(x)$ itself  is in $(-\zeta, 0)$ then $\phi^{2}(x)>0$. In any case, there is $m_x \geq 2$ such that 
$\phi^{m_{x}} (x)>0$ whenever $x \in (0,\zeta)$ (see Figure~\ref{graph-phi}(a)). Denoting $\phi^{m_{x}} (x)$ by $x_1$,  it follows from (3) of this lemma again that $\phi^{n_{x_1}}(x_1) \in (0, \zeta]$ for some $n_{x_1} \geq 1$    whenever $x_1 >\zeta$. For $x_1 \in (0, \zeta)$,  we take $n_{x_1}=0$. Thus, for every $x \in I_\zeta$, there is $N_x \geq 2$ such that $\phi^{N_x}(x) \in I_\zeta$. 
		
\par To prove that for each real $x \notin I_\zeta$, $\phi^n (x) \in I_\zeta$ for some $n$, first note that this is true for each $x> \zeta$ by (3) of this lemma. If $x< 0$ then it is already seen in the previous paragraph that $\phi^{n}(x)>0$ for some $n\geq 1$. We are done since now $\phi^{n}(x)>\zeta$ or is in $I_\zeta$.\end{enumerate}
\end{proof}
For the proof of Theorem~\ref{equal-multiplicity}, we need to understand the common boundary of the two immediate basins.  
 Let $\gamma$ denote the common boundary of the two immediate basins $\mathcal{A}_0$ and $ \mathcal{A}_1$ of $C$ corresponding to $0$ and $1$, respectively. Since the boundary of each (invariant) immediate basin is itself forward invariant under $C$,  $\gamma$ is also forward invariant. However, since $\gamma$ is properly contained in the Julia set, it can not be backward invariant. However,   it is true in a restricted sense as described below.
 \begin{lem}
 	If $\gamma$ is the common boundary  $\partial \mathcal{A}_0 \cap \partial \mathcal{A}_1$ and $w \in \gamma$ then there are   two distinct points $z_{-1}, z_{-2}$ in $\gamma$ such that $C(z_{-i})=w$ for $i=1,2$.
 \label{common-boundary-backwardinvariant}
 \end{lem}
 \begin{proof}
 	 Let $w \in \gamma$. Then it follows from Lemma~\ref{phi}(2) that for every point $w \in L$, there are exactly two points $z_{-1}, z_{-2} \in L$
 	such that $C (z_{-i})=w$ for $i=1,2$.  There are three points on the boundary of $\mathcal{A}_0$ which are mapped to $w$ by $C$. This is because $C: \mathcal{A}_0 \to \mathcal{A}_0$ is a proper map with degree three. None of these three points can be with real part strictly greater than $0.5$ as $L$ separates $\mathcal{A}_0$ from  $\mathcal{A}_1$. In other words, the real part of each such point is less than or equal to $0.5$. Arguing similarly for $\mathcal{A}_1$, it is seen that the three points on the boundary of $\mathcal{A}_1$ that are mapped to $w$ by $C$  have their real parts greater than or equal to $0.5$. Since there are four pre-images of any point under $C$, the points  $z_{-1}, z_{-2} \in L$ must therefore lie on the common boundary  $\gamma$.
 \end{proof}
%

\begin{proof}[Proof of Theorem~\ref{equal-multiplicity}]Recall that $L$ is the vertical line passing through the pole $0.5$.	
	Since the Fatou set of $C$ is the union of the two attracting basins and $C (L) \subset L$, this line $L$ is contained in the Julia set of $C$, separating the two immediate basins.
	In other words, the common boundary $\gamma  = \partial \mathcal{A}_0 \cap    \partial \mathcal{A}_1$  is a subset of $L$. Next, we  show that $L \subseteq \gamma$.
	\par 
	
Let  $w \in L$ and  $J \subset L$ be an open interval containing $w$. Since the map $C : L \to L$ is conformally conjugate to $\phi: \mathbb{R} \to \mathbb{R}$, it follows from Lemma~\ref{phi}(4) that   $C ^{n_k} (w) \in J_\zeta=(0.5, 0.5+i \zeta)$ for infinitely many values of $k$, unless $C^n (w)=0.5$ for some $n$.
We now show that the first possibility can not be true.

\par   Suppose that $C ^{n_k} (J)$ is strictly contained in $J_\zeta$ for all $k$. Note that $n_k \to \infty$ as $k \to \infty$, and we assume without loss of   generality that it is a strictly increasing sequence. Consider a disk $D$ containing $J_\zeta$ but not containing any point of the post-critical set of $C $ other than the pole $0.5$. This is possible as all the four critical points of $C $ other than the pole are in the Fatou set. The pole is neither a critical value nor in the forward orbit of any critical point of $C$. Therefore, all the branches of  $C ^{-n_k}$ are well-defined on $D$ for each $k$. Let $f_k$ be the branch of $C ^{-n_k}$ such that $f_k ( C ^{n_k} (J) )=J$. By the Monodromy theorem, each $f_k$ is analytically continued throughout $D$. 
	By Lemma~\ref{backward-contraction}, the family $\{f_k\}_{k>0}$ is normal on $D$ and each of its limit function is constant. Therefore, the diameter of $f_k (D)$ tends to $0$ as $k \to \infty$. However, each $f_k(D)$ contains $J$, which has positive diameter leading to a contradiction. Therefore,  there exists an $n'$ and $ z \in J$ such that $C^{n'}(z)=0.5$.
\par Note that the pole $0.5$ is on the common boundary  $\gamma$. Its two pre-images lying on $L$ are in $\gamma$ by Lemma~\ref{common-boundary-backwardinvariant}.
Applying the same lemma  to each of these pre-images, it is found that the four points whose $C^2$-image is $0.5$ are also in $\gamma$. Repeating this argument, we get that  $ z \in \gamma$. The set $\gamma$ is closed as it is the intersection of two closed sets, namely $\partial \mathcal{A}_0$ and $\partial \mathcal{A}_1$. Therefore, the point $w$, being a limit point of $\gamma$, belongs to $\gamma$. In other words, the vertical line $L$ is contained in the common boundary  $\gamma$.

\par The proof competes as $L$ is  connected.

\end{proof}
\begin{rem}
	The Julia set of $C$ is not a Jordan curve. In particular, $\mathcal{J}(C) \setminus L$ is non-empty. In fact, this set  has infinitely many  maximally connected sets.
\end{rem}

We conclude  with	
 following observations that can be made from Figure~\ref{k=m-dynamics}.
\begin{enumerate}
	
	\item   $\mathcal{J}(C_m) \to L$ as $m \to \infty$ with respect to the Hausdorff distance.
	\item For each $m,$ the components of $\mathcal{J}(C_m) \setminus L$ with the largest diameter are the ones containing the extraneous fixed points. These components intersect  the real axis. 
	
\end{enumerate}

For $k \neq m$, the pole of the Chebyshev's method applied to $z^k (1-z)^m$  is $\frac{k}{k+m}$. This pole lies on the common boundary of the two immediate basins (by Lemma~\ref{Jconnected}). 
There are some natural questions on the structure of the basins.

\begin{enumerate}
\item Is the  common boundary of the two immediate basins  connected? This appears to be the case  from the images given in Figure~(\ref{Julia_Fatou2}) for $(k,m)=(6,4)$ and $(k,m)=(3,10)$. A possible approach to prove this could involve establishing a quasi-conformal conjugacy between $C_{z^k (z-1)^m}, k \neq m$ and $C_m$. It is of course clear that $C_{z^k (z-1)^m}, k \neq m$ and $C_m$ are not conformally conjugate as they have fixed points with different multipliers.
	\item For $m>k$, the immediate basin of $1$ seems to contain the vertical line passing through the pole but not the pole itself, and the half-plane lying right to it. Is it actually true?
	\item  What happens to the immediate basin of $0$ when $k$ is fixed and $m \to \infty$?
	\item The Chebyshev's method applied to any polynomial with two distinct roots is a quartic rational map. Indeed,  every Chebyshev's method with degree four arises from a polynomial with two distinct roots. Such rational maps have two attracting  and three repelling fixed points.  It might be interesting to find a minimal set of assumptions under which a quartic rational map is conjugate (or equal) to the Chebyshev's method applied to a polynomial with two distinct roots.
\end{enumerate}
A proof or a counterexample to any of these would be interesting to know.

 \section{Declarations}
 \subsection{Author Contribution} Both authors contributed equally.
 \subsection{Funding} Pooja Phogat is funded by a Senior Research Fellowship (Grant number: 09/1059(0031)/2020-EMR-I) provided by the Council of Scientific and Industrial Research, Govt. of India. Funding is not applicable to the other two authors.
 \subsection{Conflicts of interest/Competing interests}
 Not Applicable.
 \subsection{Data Availability statement} Data sharing not applicable to this article as no datasets were generated or analyzed during the current study.
 \subsection{Code availability} Not Applicable.

\end{document}